\documentclass[12pt]{article}

\usepackage{graphicx}
\usepackage{amsmath}
\usepackage{amsfonts}
\usepackage{amsthm}
\usepackage{amssymb}
\usepackage{color}
\usepackage{hyperref}
\usepackage{verbatim}

\newtheorem{theorem}{Theorem}
\newtheorem{lemma}[theorem]{Lemma}
\newtheorem{corollary}[theorem]{Corollary}

\theoremstyle{definition}

\begin{document}

\title{On the Structure of nil-Temperley-Lieb Algebras of type A}

\author{Niket Gowravaram \and Tanya Khovanova}
\date{}
\maketitle

\begin{abstract}
We investigate nil-Temperley-Lieb algebras of type A. We give a general description of the structure of monomials formed by the generators. We also show that the dimensions of these algebras are the famous Catalan numbers by providing a bijection between the monomials and Dyck paths. We show that the distribution of these monomials by degree is the same as the distribution of Dyck paths by the sum of the heights of the peaks minus the number of peaks.
\end{abstract}


\section{Introduction}

This research was done as a part of PRIMES: a program that helps high-school students conduct research. The project was suggested by Prof.~Alexander Postnikov.

Postnikov suggested the following construction to generate an algebra from a graph \cite{Postnikov}:

Given a simple graph $G$, we construct a unital algebra associated with it. For every vertex we have a generator. The square of each generator is 0. Suppose $x$ and $y$ are two generators. If their corresponding two vertices are not connected by an edge, then the corresponding generators commute: $xy=yx$. If there is an edge connecting the vertices, then the following relations hold: $xyx=0$ and $yxy=0$.

Postnikov called these algebras \textit{XYX algebras}. The algebras were inspired by the idea of fully-commutative elements of Coxeter groups \cite{Stembridge1995}.

When we discovered that the dimensions of XYX algebras related to path are Catalan numbers, we discussed it with Prof.~Richard Stanley, the author of a recent book on Catalan Numbers \cite{Stanley2015}. Richard Stanley recognized the algebras we studied as nil-Temperley-Lieb algebras. The fact that the dimensions of nil-Temperley-Lieb algebras of type A are Catalan numbers is widely known \cite{Benkart2015,Fomin1995, Stanley2015}.

In our research we also calculated the distribution of the dimensions of these algebras by degree and found a bijection to a statistic of Dyck paths.

We later discovered that our statistic is equivalent to counting inversions in 321-avoiding permutations \cite{Cheng2012}.

In this paper, we provide some basic definitions and examples of the dimensions of the algebras for small cases in Section~\ref{sec:examples}. We go on to describe the structure of the monomials in the algebra and divide them into descending runs of generators in Section~\ref{sec:structure}. We then provide a bijection between the monomials and Dyck paths, by corresponding the descending runs of generators to peaks in the Dyck paths in Section~\ref{sec:Dyck}. In Section~\ref{sec:thm} we show that the number of monomials of degree $d$ in the algebra corresponding to the path graph $P_n$ is equal to the number of Dyck paths of length $2n$ such that the sum of the heights of the peaks minus the number of peaks is $d$. In Section~\ref{sec:avoiding} we explain the connection to 321-avoiding permutations. In Section~\ref{sec:other} we conclude with  some corollaries of our bijection and some interesting properties.

\section{Definitions and Small Examples}\label{sec:examples}

In this paper, we deal only with path graphs. Let $P_n$ be the path graph with $n$ vertices. From now on we will number the vertices along the path, and represent the generators as $x_i$, where $1 \leq i \leq n$ and $n$ is the length of the path. We denote the path graph with $n$ vertices as $P_n$.

The \textit{nil-Temperley-Lieb algebra} corresponding to $P_n$ is a unital algebra generated by $x_i$ and the following relations:

\begin{itemize}
\item $x_i^2 = 0$
\item $x_i x_j = x_jx_i$, if $|i-j| > 1$
\item $x_i x_{i+1} x_i = 0$, for $1 \leq i < n$
\item $x_{i+1} x_{i} x_{i+1} = 0$, for $1 \leq i < n$.
\end{itemize}

We start with small examples of $P_0$, $P_1$, and $P_2$. Here are the bases and dimensions of the corresponding algebras:

\begin{itemize}
\item $P_0$. Basis: 1. Dimension 1.
\item $P_1$. Basis: 1, $x_1$. Dimension 2.
\item $P_2$. Basis: 1, $x_1$, $x_2$, $x_1x_2$, and $x_2x_1$. Dimension 5.
\item $P_3$. Basis: 1, $x$, $x_2$, $x_3$, $x_1x_2$, $x_1x_3$, $x_2x_1$, $x_2x_3$, $x_3x_2$, $x_1x_2x_3$, $x_1x_3x_2$, $x_2x_1x_3$, $x_3x_2x_1$, $x_2x_1x_3x_2$. Dimension 14.
\end{itemize}

We see that the dimensions of these algebras form one of the most famous sequences in mathematics: the Catalan numbers \cite{oeiscatalan,Stanley2015}. 

We would like to look at the structure of the monomials in this algebra in more detail.

\section{Structure of the monomials}\label{sec:structure}

The algebra corresponding to a path as a vector space is generated by monomials. Each of the monomials is a product of generators corresponding to vertices. 

We call two monomials \textit{equivalent} if they correspond to the same element in the basis. We call a monomial \textit{reducible} if it is equivalent to the zero monomial. We call a monomial an \textit{nTL-monomial} if it is an irreducible monomial that is lexicographically smallest in its equivalency class.

Consider some examples: monomials $x_3x_1x_2$ and $x_1x_3x_2$ are equivalent and the latter is an nTL-monomial. The monomials $x_3x_1x_2x_1$ and $x_1x_3x_2x_1$ are equivalent and reducible.

The decreasing strings of letters in an nTL-monomial are very important in our study. Let us call the longest consecutive sequence of letters in an  nTL-monomial that is decreasing a \textit{decreasing run}. 

\begin{lemma}\label{thm:decreasing}
The indices in a decreasing run go down by one. 
\end{lemma}

\begin{proof}
If an nTL-monomial contains a substring $x_ix_j$ with $j < i-1$, then it is not in its earliest lexicographic form as we can switch $x_i$ and $x_j$.
\end{proof}

Call the first generator of a decreasing run \textit{a peak}. Thus we can write down an nTL-monomial as a sequence of pairs $(p_i,r_i)$, where $p_i$ is the index of the $i$-th peak and $r_i$ is the length of the $i$-th run.

For example, the nTL-monomial $x_3x_2x_1x_4x_3$ has two runs and can be represented as two pairs $(3,3)$, $(4,2)$.

Notice that the sum of the runs' lengths, $\sum_i r_i$, is the degree of the nTL-monomial.

We want to find the properties of the nTL-monomials and the constraints under which they are in one-to-one correspondence with the sequences of pairs. The first lemma does not need a proof.

\begin{lemma}
The length of a run of an nTL-monomial cannot be greater than the index of its peak: $r_i \leq p_i$.
\end{lemma}

Here is the next constraint. 

\begin{lemma}\label{thm:peaksincreasing}
The indices of the peaks of an nTL-monomial are in increasing order: $p_i < p_{i+1}$.
\end{lemma}

\begin{proof}
Suppose $x_i$ and $x_j$ are two consecutive peaks in an nTL-monomial. The run that starts with $x_i$ ends in $x_k$, where $k < j$, otherwise, $x_j$ is not a peak. If $j \leq i$, then $x_j$ also is in the run, and we will be able to reduce the nTL-monomial to zero.
\end{proof}

Let us define a \textit{valley} as the last generator in a run. We denote the $i$th valley as $v_i$. As the indices in a run go down by one, we can see that $v_i = p_i-r_i+1$.

The following lemma can be proved the same way as Lemma~\ref{thm:peaksincreasing}, or by invoking the symmetry between runs and valleys.

\begin{lemma}\label{thm:valleyssincreasing}
The indices of the valleys of an nTL-monomial are in increasing order: $v_i < v_{i+1}$.
\end{lemma}

The following theorem shows that there are no more constraints.

\begin{theorem}\label{thm:bijectionpairs}
If there is a sequence of pairs satisfying the constraints defined in Lemmas~\ref{thm:decreasing}-\ref{thm:valleyssincreasing}, then there exists an nTL-monomial corresponding to it.
\end{theorem}

\begin{proof}
First, any nTL-monomial constructed by a sequence of pairs following the given constraints cannot be permuted to create a monomial that is lexicographically before our nTL-monomial. Now, we wish to show that the constructed nTL-monomial cannot be reduced to zero. Take any element $x_i$ anywhere in our nTL-monomial. If this element is not a peak or a valley, it will be impossible to move it, because it is surrounded by elements that it does not commute with. So, we will not be able to form the patterns $x_ix_i$, $x_ix_{i+1}x_i$, or $x_{i+1}x_ix_{i+1}$, and therefore we will be unable to reduce our nTL-monomial to zero.

If our chosen element $x_i$ is a valley, then we will be able to move it in the forward direction until we reach an element $x_{i+1}$. However, since the valleys are strictly increasing and the valley will have the lowest index of all the elements in the run, there will be no more $x_i$ elements. So, it will be impossible to create the pattern $x_ix_{i+1}x_i$. The argument for if the chosen element is a peak is essentially the same, except for the fact that the peak can only move backwards instead of forwards in the  nTL-monomial.

So, any monomial constructed from pairs following these constraints is an  nTL-monomial.
\end{proof}

\section{Dyck Paths}\label{sec:Dyck}

Now we want to build a bijection between the  nTL-monomials and Dyck paths. We will use the Dyck paths represented as mountain ranges: the up direction is NE and the down direction is SE. In addition, the range never crosses the $x$-axis and ends on the $x$-axis.

We correspond an nTL-monomial to a Dyck path in the following manner. Suppose there is a hill with the top located at coordinates $(a,b)$. Then we correspond it to a run of length $b-1$ starting with a letter $x_k$, where $k= (a+b-2)/2$. That is, we correspond it to a pair $((a+b-2)/2,b-1)$.

In particular, the last index of the run is $(a+b-2)/2 -b+1= (a-b)/2$. In addition, all the small hills of height 1 are ignored.

Figure~\ref{paths} shows 14 Dyck paths of length 8. The caption under each picture describes the corresponding nTL-monomials.

\begin{figure}[htbp]
\begin{center}
\begin{tabular}{ccccc}
\includegraphics[scale=0.15]{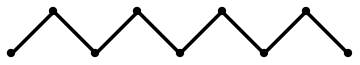} & 
\includegraphics[scale=0.15]{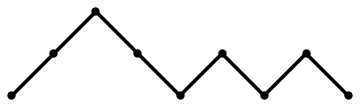} &
\includegraphics[scale=0.15]{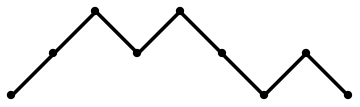} &
\includegraphics[scale=0.15]{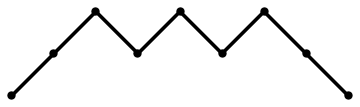} &
\includegraphics[scale=0.15]{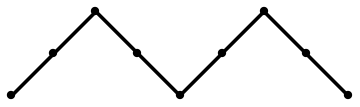} \\
1 & $x_1$ & $x_1x_2$ & $x_1x_2x_3$ & $x_1x_3$ \\
\includegraphics[scale=0.15]{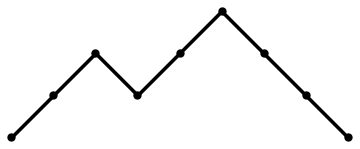} & 
\includegraphics[scale=0.15]{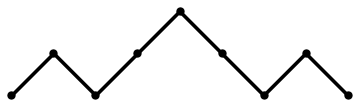} &
\includegraphics[scale=0.15]{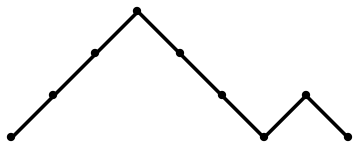} &
\includegraphics[scale=0.15]{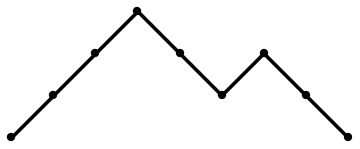} &
\includegraphics[scale=0.15]{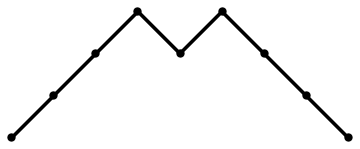} \\
$x_1x_3x_2$ & $x_2$ & $x_2x_1$ & $x_2x_1x_3$ & $x_2x_1x_3x_2$
\\

\includegraphics[scale=0.15]{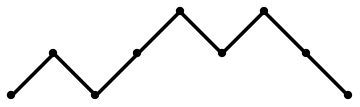} & 
\includegraphics[scale=0.15]{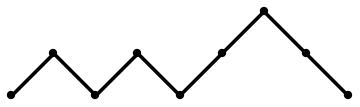} &
\includegraphics[scale=0.15]{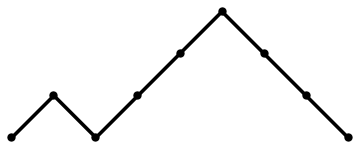} &
\includegraphics[scale=0.15]{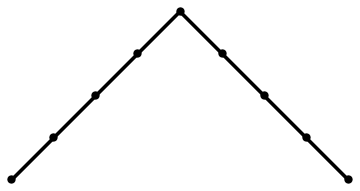} & \\
$x_2x_3$ & $x_3$ & $x_3x_2$ & $x_3x_2x_1$ &
\\
\end{tabular}
\caption{The Dyck paths of length eight and their corresponding nTL-monomials} \label{paths}
\end{center}
\end{figure}

\begin{theorem}
The described correspondence is a bijection.
\end{theorem}

\begin{proof}
Let us show that a Dyck path generates the sequence of pairs with correct constraints.

\textit{The length of the run}. We need to show that $b-1 \leq (a+b-2)/2$. This is equivalent to $b \leq a$, which is true as the Dyck path does not go over the line $x=y$.

\textit{The peaks are increasing}. Suppose $(a_i,b_i)$ and $(a_{i+1},b_{i+1})$ are two consecutive non-trivial hills. Then $a_{i+1}+b_{i+1} > a_i+b_i$. This is the same as saying that the line corresponding to the right side of the next hill is above the line corresponding to the right side of the previous hill. This is exactly the condition $p_{i+1} > p_i$.

\textit{The valleys are increasing}. The valley $v_i = p_i - r_i +1 = (a_i+b_i -2)/2 - b_i +1 = (a_i-b_i)/2$. The statement is true because the line corresponding to the left side of the next hill is below the line corresponding to the left side of the previous hill.

This shows that the map is onto. What is left to see is that we cannot get the same nTL-monomial from two different Dyck paths. A Dyck path is uniquely defined by its non-trivial hills. And different sequences of non-trivial hills produce different sequences of pairs.
\end{proof}

The reversal of this bijection describes how we correspond a Dyck path to a sequence of pairs. For every pair $(p_i,r_i)$ we draw a mountain top with coordinates $(p_i+r_i+1,r_i+1)$. Furthermore, in order to get the complete picture, we superimpose the different peaks corresponding to the different runs on top of each other and if we have a flat space, we replace that space with hills of height one.

\section{What follows from the bijection}\label{sec:thm}

Given that we found a bijection with Dyck paths, the Catalan numbers appear immediately.

\begin{corollary}
The dimension of the algebra corresponding to the path graph $P_n$ is the $(n+1)$-st Catalan number.
\end{corollary}

The bijection also allows us to describe the distribution of nTL-monomials by degree.

\begin{theorem}
The number of nTL-monomials of degree $d$ in the algebra corresponding to the path graph $P_n$ is equal to the number of Dyck paths of length $2n$ such that the sum of heights of peaks minus the number of peaks is $d$.
\end{theorem}

\begin{proof}
We only need to mention that trivial hills contribute zero to the sum.
\end{proof}

For example, the distribution by degree for $P_3$ is: 1, 3, 5, 4, and 1, with the total dimension of 14.

\section{321-avoiding permutations}\label{sec:avoiding}

The distribution by degree appears in the On-Line Encyclopedia of Integers Sequences as sequence  A140717 \cite{oeis140717}. It is defined there through 321-avoiding permutations having a given inversion number.

The connection between nil-Temperley-Lieb algebras and permutations is well known \cite{Billey1993,Cheng2012,Pesiri2011}. We correspond to an element $x_i$ a permutation switching elements $i$ and $i+1$. This element is denoted as $s_i$. Then a run $x_{i-1}\ldots x_{j}$ corresponds to the product $s_{i-1} \ldots s_{j}$ that is a permutation that swaps $i$ and $j$. This swap is usually denoted as the transposition $(i,j)$.

A well-known property of the 321-avoiding permutations is that each has a unique reduced expression of the form $(a_1,b_1),\ldots,(a_k,b_k)$, where $a_k > a_{k-1} > \ldots > a_1$ and $b_k > b_{k-1} > \ldots > b_1$ \cite{Rhoades2005}. The $a_i$ and $b_i$ represent the peaks and valleys, respectively, in our nTL-monomials and we know that the peaks and valleys in our monomials must be increasing. So, from each monomial we can create a unique 321-avoiding permutation. In paper \cite{Cheng2012} the number of inversions of 321-avoiding permutations is explained in terms of Dyck paths.

\section{Other properties}\label{sec:other}

Here we would like to mention some nice properties that follow from our bijection. 

Suppose an element $x_i$ is the $j$-th smallest/largest index in an nTL-monomial.

\begin{lemma}\label{thm:onex}
The elements $x_i$ appears not more than $j$ times in an nTL-monomial.
\end{lemma}

\begin{proof}
The element $x_i$ can only appear in a run which ends at a valley that is not exceeding $i$. The lemma follows from the fact that valleys are increasing.
\end{proof}

We also want to provide some connections between the shapes of the Dyck paths and the corresponding nTL-monomials.

\begin{lemma}
If a Dyck path touches the baseline $y=0$ at $x=k$, then the element $x_{k/2}$ is missing from the corresponding nTL-monomial.
\end{lemma}

\begin{proof}
Suppose that the path hits the baseline. Then, we can split our path into two separate paths and find the nTL-monomial for each path, where we shift the set of letters for the second nTL-monomial.
\end{proof}

If our path does not hit the baseline, then we can move the baseline up by one step and find the nTL-monomial corresponding to this new, shorter path. With this nTL-monomial, we extend each peak by adding one letter to the left of the nTL-monomial corresponding to the apex.

There is also a method to find the number of times each element appears in an nTL-monomial/Dyck path. First, we extend all hills so they hit the baseline on both sides. With this modified picture, we can find the number of times each element $x_k$ appears in the nTL-monomial.

\begin{lemma}
The number of times the  element $x_k$ appears in the nTL-monomial is the number of intersections that are not on the baseline between the left sides of the extended hills, which do not include the peaks, and the line $x+y=2k$.
\end{lemma}

\section{Acknowledgements}

We would like to thank the MIT PRIMES program for supporting this research. We also want to thank Prof.~Postnikov for suggesting the project and Prof.~Stanley for helpful suggestions and references.

\end{document}